\documentclass[journal,twoside,web]{ieeecolor}

\usepackage{generic}

\usepackage{amsmath} 
\usepackage{amssymb}  
\usepackage{bbm}
\usepackage{stmaryrd}
\usepackage{color}
\usepackage{tikz}
\usepackage{cite}
\usepackage{lipsum}
\usepackage{todonotes}
\usepackage{graphics} 
\usepackage{epstopdf}
\usepackage{epsfig}

\newcommand*\diff{\mathop{}\!\mathrm{d}t}

\usetikzlibrary{patterns,calc,angles,quotes,positioning,datavisualization}
\tikzset{
  treenode/.style = {shape=rectangle, rounded corners,
                     draw, align=center,
                     top color=white, bottom color=blue!20},
  root/.style     = {treenode, font=\Large, bottom color=red!30},
  env/.style      = {treenode, font=\ttfamily\normalsize},
  dummy/.style    = {circle,draw,minimum size = 0.6cm,pattern=crosshatch, pattern color = black!50},
  graph/.style    = {circle,draw,minimum size = 0.6cm,pattern=dots, pattern color = black!25},
  action/.style   = {circle,draw,minimum size = 0.6cm,pattern=crosshatch, pattern color = black!25},
}

\newtheorem{lemma}{Lemma}
\newtheorem{theorem}{Theorem}

\newtheorem{remark}{Remark}
\newtheorem{proposition}{Proposition}

\newcommand{\lap}[0]{{\mathcal{L}}}

\newcommand{\graph}[0]{{\mathcal{G}}}

\newcommand{\Nodes}{\mathcal{N}}
\newcommand{\Verts}{\mathcal{V}}

\newcommand{\Edges}{\mathcal{E}}

\newcommand{\Htwo}{{\mathcal{H}_2}}

\newcommand{\reals}{{\mathbb{R}}}

\newcommand{\D}{{\mathbf{D}}}
\newcommand{\T}{{\mathbf{T}}}
\newcommand{\E}{{\mathbf{E}}}
\newcommand{\W}{{\mathbf{W}}}
\newcommand{\R}{{\mathbf{R}}}
\newcommand{\X}{{\mathbf{X}}}

\usepackage{bbold}
\usepackage{mathtools}
\usepackage{soul}
\usepackage{tikz}
\usepackage{color}
\usepackage{tcolorbox}
\usepackage{float}
\usepackage{subfig}
\usepackage{algorithm2e}
\newlength\mylen

\makeatletter
\newcommand{\removelatexerror}{\let\@latex@error\@gobble}
\pretocmd\@bibitem{\color{black}\csname keycolor#1\endcsname}{}{\fail}
\newcommand\citecolor[1]{\@namedef{keycolor#1}{\color{blue}}}
\makeatother
\makeatletter
\renewcommand{\@algocf@capt@plain}{above}
\makeatother

\usepackage{tikzscale}
\mathtoolsset{showonlyrefs}
\usepackage{textcomp}

\begin{document}
\title{Performance and Design of Consensus on Matrix-Weighted and Time Scaled Graphs}
\author{Dillon R. Foight, \IEEEmembership{Student Member, IEEE,} Mathias Hudoba de Badyn, \IEEEmembership{Member, IEEE,}\\ and
        Mehran Mesbahi, \IEEEmembership{Fellow, IEEE} 
\thanks{Manuscript received November 22, 2019; revised March 18, 2020 and May 12, 2020; accepted May 26, 2020. This work was supported by the National Science Foundation (grant number DGE- 1762114), the National Sciences and Engineering Research Council of Canada (funding reference number CGSD2-502554-2017), the U.S. Army Research Laboratory and the U.S. Army Research Office (contract number W911NF-13-1-0340), and the U.S. Air Force Office of Scientific Research (grant number FA9550-16-1-0022). \textit{(D. Foight and M. Hudoba de Badyn are co-first authors.) (Corresponding author: D. Foight.)} A preliminary version of this work appears in the Proceedings of the 58th IEEE Conference on Decision and Control~\cite{Foight2019a}. }%
\thanks{The authors are with the William E. Boeing Department of Aeronautics and Astronautics at the University of Washington, Seattle, WA 98195, USA.
  MHdB is now with the Automatic Control Laboratory at ETH Z\"{urich}, 8092 Z\"{u}rich, Switzerland. e-mails: \texttt{\{dfoight,hudomath,mesbahi\}@uw.edu}. \textcopyright IEEE 2020}
}

\markboth{Transactions on Control of Networked Systems,~Vol.~xx, No.~x, August~20xx}%
{Foight, Hudoba de Badyn, \& Mesbahi: Time Scale and Matrix Weight Design for Consensus}

\maketitle

\begin{abstract}
  In this paper, we consider the $\Htwo$-norm of networked systems with multi-time scale consensus dynamics and vector-valued agent states.
  This allows us to explore how measurement and process  noise affect consensus on matrix-weighted graphs by examining edge-state consensus.
  In particular, we highlight an interesting case where the influences of the weighting and scaling on the $\Htwo$ norm can be separated in the design problem.
  We then consider optimization algorithms for updating the time scale parameters and matrix weights in order to minimize network response to injected noise.
  Finally, we present an application to formation control for multi-vehicle systems.
\end{abstract}


\IEEEpeerreviewmaketitle

\section{Introduction}

\IEEEPARstart{D}{ynamical} systems operating over networks appear in many natural and cyber-physical systems. A popular model of such dynamic processes is \emph{consensus}, which has been used for a variety of control and estimation applications, such as multi-agent systems~\cite{Chen2013a,Saber2003,Tanner2004}, robotics,~\cite{Joordens2009,Hudobadebadyn2018}, and distributed Kalman filtering~\cite{Olfati-Saber2005}. A natural question is how the underlying network topology affects the behavior of the dynamics operating over the network, motivated by the fact that notions of performance and control can be directly related to graph theoretic properties of the network. Of particular interest for this work is the $\Htwo$ system norm, which for networked dynamical systems can be interpreted as a measure of how input energy is attenuated over the network, or how noise drives deviations from the natural consensus state~\cite{Siami2014a}. 

In light of these interpretations, several works have characterized the $\Htwo$ performance for consensus networks. In~\cite{Chapman2015,9031557}, the performance of leader-follower networks is considered, and algorithms for rewiring and reweighting the network for optimal noise rejection are discussed. Similarly,~\cite{Bamieh2012,Patterson2010,Patterson2014} have utilized the $\Htwo$-norm as a measure of coherence in networks and considered problems such as local feedback laws and leader selection to promote coherence. Most relevant to the present contribution, the works~\cite{Zelazo2011a,Zelazo2013} investigated the impact of cycles on the $\Htwo$ performance of noise-driven consensus networks. Examining networks under noise inputs is especially important for network security and resiliency~\cite{Pasqualetti2013}, which motivates considering the minimization of the $\Htwo$ system norm in order to promote resilience to external noise inputs (which could be adversarial in nature).

A common simplifying assumption within the literature is that the agent dynamics are identical single or double integrators.
An extension of this simplified model is to consider the case where individual agents' states evolve at differing rates, motivated by similar formulations in areas such as electrical networks~\cite{Dorfler2018} and power networks with generator inertia~\cite{Chakrabortty2013}. The analysis of multi-scale problems has historically offered techniques for formal description and controller synthesis for complex systems~\cite{Kokotovic1999}; analysis of a multi-scale consensus model can increase the applicability of the consensus protocol to a wider range of real-world systems. As such, there is a growing body of literature addressing the complications that arise from the integrating multiple time scales into consensus, starting with the discussion of the consensus value for multi-rate integrators in~\cite{Olfati-Saber2004}. Issues such as convergence \cite{Pedroche2014}, stability \cite{Chapman2016,Awad2018}, controller design~\cite{Rejeb2016,Rejeb2018}, formation control~\cite{Deghat2018}, as well as single-influenced consensus performance~\cite{Foight2019} have since been addressed for such multi-scale networks. Considerations of optimizing the time scale parameters have been rare, possibly in part due to complications such as bilinear matrix equations that can naturally arise in multi-time scale formulations~\cite{Chapman2016}. An advantage of the noise driven edge consensus formulation presented here is that, under specific assumptions, the $\Htwo$ performance is characterized in a form that allows for streamlined optimization formulations.

Additionally, most of the existing literature on control and performance analysis for consensus consider the case where each node in the network has a \emph{scalar} state, or can be decomposed as such.
Recently, several extensions to the case of \emph{vector}-valued node states have been proposed; such consensus protocols are characterized by having graphs with \emph{matrix-valued edge weights}.
System-theoretic properties of matrix-valued weighted consensus networks, as well as applications to bearing-constrained formation control have been considered in~\cite{Trinh2017}.
Matrix weighted graphs have also been utilized in applications such as estimation~\cite{Barooah2008,Lee2016a}, spacecraft formation control~\cite{Ramirez2009}, the control of coupled oscillators~\cite{Tuna2016}, and opinion dynamics~\cite{Friedkin2016,Parsegov2017}.

In this paper, we consider design problems networks using the $\Htwo$ system norm as a metric of network's resiliency.
In particular, we examine consensus on matrix-weighted and time scaled networks, with process and measurement noise.
Drawing from the work in~\cite{Zelazo2011a}, we transform the general consensus problem to one over the edge states, and consider $\Htwo$-optimal design of the agent time scales and matrix edge weights.
The main contributions of the paper are the formulation of the consensus problem for matrix-weighted networks with time scales, a transformation yielding the dynamics of the edge states for this consensus problem, a method for separating the contributions of edge weighting and node time scales on the $\Htwo$ performance, design problems for time scale and edge weight assignment, and an application to formation control.

The paper is organized as follows.
In \S\ref{sec:math-prel}, we outline the necessary notation and graph theory used in the paper, and introduce the problem setup in \S\ref{sec:setup}.
The main results are divided into the formulation of the $\Htwo$ performance metric in \S\ref{sec:h2-results}, followed by design problems for time scale and edge weights in \S\ref{sec:ts-design} and \S\ref{sec:gu-design}. We apply the distributed design problems to a flocking model on second-order consensus in \S\ref{sec:example}. For readability, we relegate some technical proofs to the Appendix. 


\section{Mathematical Preliminaries} \label{sec:math-prel}

Here, we provide a brief overview of the notation and terminology used throughout the paper, as well as relevant graph theoretic concepts. Column vectors are denoted as $x \in \reals^n$. Special vectors include the vector of all ones (zeros), $\mathbf{1}$ ($\mathbf{0}$), the vector of diagonal elements in a matrix, $\mathrm{diag}(M)$, and Euclidean basis vectors, $e_i \in \reals^n$, where the $i$ denotes the index of the non-zero element. Matrices will be denoted as $M \in \reals^{m \times n}$. The ($k\times k$)  identity matrix will be denoted by $I_k$. The set of positive-(semi) definite matrices will be denoted by $\mathcal{S}_{++}^n$ ($\mathcal{S}_+^n$). The Kronecker product of two matrices $A,B$ is denoted by $A\otimes B$, and the operation $\mathrm{Blkdiag}(A_1,\dots,A_n)$ yields the matrix which has the matrices $A_1,\dots, A_n$ on its diagonal. For matrices $A$ and $B$, $A \preceq B$ implies $B-A \in \mathcal{S}_+$. Time-dependent quantities will be denoted as $x(t)$.

This paper considers dynamics governed by the interconnections of multi-rate, single integrator agents over connected, matrix-weighted communication graphs. In this formulation, we can consider a graph object defined by $\graph= \left(\Verts,\Edges\right)$, where $\Verts$ is the set of agents (nodes), and $\Edges$ is the set of edges. Associated with the graph are $\cal W$, a set of matrix edge weights, and $\cal T$, a set of time scaling factors for agents' states.

Individual agents states are vector-valued, $x\in\mathbb{R}^k$, and each agent will be indexed by subscripts, e.g. $\nu_i \in \Verts$ to represent the $i$-th agent where $1\leq i \leq |\Verts|$.
If $(i,j) \in \Edges$, the $i$-th and $j$-th agents are connected by an edge ($i \sim j$), and they are referred to as adjacent agents. For a given agent, $i$, $N(i) = \{j\ |\ i\sim j\ \forall j \in \Verts\}$ denotes the neighbors of $i$, and $\text{deg}(\nu_i) = |N(i)|$ denotes the unweighted degree of $i$. The $k$ values comprising an agent's state will be referred to as \textit{substates}, and the $j$-th substate of the $i$-th agent is denoted as $x_{i,j}$. 

As a consequence of considering vector-valued agent states, edges between agents are matrix-valued, which allows for a notion of dynamical coupling between neighboring agent states. Such matrix-valued weights will be denoted $W_{e}\in\mathcal{S}_{++}^k$, and so $\mathcal{W} = \{W_e~|~ e\in \Edges\}$. The \emph{weight matrix} $\bf W$ is a $k|\Edges|\times k|\Edges|$ blockwise diagonal matrix containing the weights $W_{ij}$ of each edge $e$. As in graphs with scalar agent states, the edge set can be ordered by a mapping, $\kappa(\cdot)$, such that $l = \kappa(ij)$ if and only if $(i,j) \in \Edges$. By this mapping, we can denote the weight on edge $\kappa(ij)$ by $W_l$ or $W_{ij}$, interchangeably. Furthermore, we also assume that each individual substate of each agent can operate on an independent time scale. Thus, for each node $i$ and set of corresponding time scales $T_i=\{\epsilon_{i,1},\dots,\epsilon_{i,k}\}\in\mathcal{T}$, we associate the \emph{time scale matrix} $E_i = \mathrm{diag}(\epsilon_{i,1},\dots,\epsilon_{i,k})\in\mathcal{S}^k_{++}$. Note that the positive-definiteness of $E_i$ is equivalent to requiring that each $\epsilon_{i,k} >0$.

The \emph{incidence matrix} $D(\graph)$ is a $|\Nodes|\times|\Edges|$ matrix with rows and columns indexed by the nodes and edges of $\graph$, respectively.
For each edge $l:=(i,j)$, where $i$ is the tail and $j$ is the head, $D(\graph)_{il} = 1$ and $D(\graph)_{jl}=-1$, and we denote the edge vector for edge $(i,j)$ by $a_{ij}$ (column of $D(\graph)$).
If $\graph$ is undirected, by convention we write that $D(\graph)_{il} = 1$ and $D(\graph)_{jl}=-1$ for $i>j$. For the formulation of matrix-valued weights, we define $\mathbf{D}(\graph) \triangleq D(\graph)\otimes I_k$ and $\mathbf{a}_{ij} \triangleq a_{ij}\otimes I_k$.
The \emph{weighted graph Laplacian} $\lap_w$ of an undirected graph $\graph$ can be defined thusly  as $\lap_w \triangleq \mathbf{D}(\graph)\mathbf{W}\mathbf{D}(\graph)^T = \sum_{ij\in\Edges} \mathbf{a}_{ij}W_{ij} \mathbf{a}_{ij}^T$.
Equivalently, it can be defined blockwise with the $k\times k$ block whose rows are associated with the $i$th node and whose columns are associated with the $j$th node given by $\sum_{j\in \mathcal{N}_i} W_{ij}$ if $i=j$, $- W_{ij}$ if $(i,j) \in \Edges$, and $\mathbf{0}_{k\times k}$ if $(i,j) \notin \Edges$.
%
%

\section{Network Models}
\label{sec:setup}
    
In this section, we will describe a general formulation for consensus over a communication network with positive-definite edge weighting and agent time scaling, with a model for measurement and process noise. The scaled consensus problem is derived from considering a group of $n$ multi-rate integrators~\cite{Olfati-Saber2004}, with zero-mean Gaussian process noise, $\omega_i(t)$ such that $\mathbb{E}\left[\omega_i(t) \omega_i(t)^T\right] = \Omega_i,~\forall i \in \Nodes$,
\begin{align}
\begin{bmatrix}
  \epsilon_{i,1} \dot{x}_{i,1} \\
  \vdots \\
  \epsilon_{i,k} \dot{x}_{i,k} 
\end{bmatrix}=
  E_i \dot{x}_i(t) =  u_i(t) + \omega_i(t), \label{eq:consensus}
\end{align}
where $x_i$ is the vector state of the $i$-th agent, $E_i$ is $\mathrm{diag}[\epsilon_{i,1},\dots,\epsilon_{i,k}]$, and $u_i$ is the control input.

Suppose that communication between agents $i$ and $j$ is corrupted by zero-mean Gaussian noise $v_{ij}$, and let $v_i$ denote the sum of all noise inputted into agent $i$ from the connections to its neighbors in $N(i)$. Without loss of generality, we can assume that the covariance of $v_i$ is given by $\mathbb{E}[v_i(t)v_i(t)^T] = \Gamma_i$. A weighted, decentralized feedback controller, with noise, that seeks to bring agents into consensus is given by,
\begin{align}
u_i(t) & = \sum_{j\in N(i)} \left[ \mathbf{W}_{ij} (x_j(t) - x_i(t)) + v_{ij}(t)\right] \nonumber \\
\mathbf{u}(t) & = -\mathbf{D}(\graph) \mathbf{W} \mathbf{D}(\graph)^T \mathbf{x}(t) + \mathbf{D}(\graph) \mathbf{v}(t), \label{eq:control}
\end{align}

\noindent where $\mathbf{W}$ is the block-diagonal matrix of edge weights with properties detailed in \S\ref{sec:math-prel}.
The vector $\mathbf{u}$ is the stacked vector of control vectors $u_i$, and $\mathbf{v}$ is the stacked vector of all measurement noises. Applying~\eqref{eq:control} to the system~\eqref{eq:consensus} with appropriate dimensions gives the general, time scaled and matrix weighted consensus problem with process and measurement noise,
\begin{equation} \label{eq:model}
\begin{aligned}
    \dot{\mathbf{x}}(t) & = -\mathbf{E}^{-1} \lap_w(\graph) \mathbf{x}(t) + \begin{bmatrix}\mathbf{E}^{-1} & -\mathbf{E}^{-1} \mathbf{D}(\graph)\end{bmatrix} \begin{bmatrix} \mathbf{\omega}(t) \\ \mathbf{v}(t) \end{bmatrix} \\
    \mathbf{z}(t) & = \D(\graph)^T \mathbf{x}(t)
\end{aligned}
\end{equation}
\noindent where $\lap_w(\graph)$ is the weighted Laplacian matrix, and $\mathbf{E} = \mathrm{Blkdiag}(E_1,\dots, E_n)$ is the full time scale matrix of $\graph$. Here, we have introduced an output $\mathbf{z}(t)$ which is used to monitor the network performance which captures the differences between the node states as they evolve. 

As noted by~\cite{Zelazo2011a,Foight2019} for scalar-valued node states over a connected graph, the zero eigenvalue (corresponding to the consensus subspace) of the Laplacian matrix precludes reasoning about the $\mathcal{H}_2$ performance of~\eqref{eq:model}. Under matrix weighting, the zero eigenvalue will have algebraic multiplicity $k$ (corresponding to the consensus subspace of each layer of substates)~\cite{Trinh2017}, so as in~\cite{Zelazo2011a} we will appeal to a similarity transformation that separates out the zero eigenvalues. We define this transformation in the following theorem.

\begin{theorem} \label{thm:similarity}
    The scaled and edge-weighted graph Laplacian for a connected graph with time scale matrix $\mathbf{E}$ and weight matrix $\mathbf{W}$, given by $\lap_{w,s} = \mathbf{E}^{-1}\mathbf{D}(\graph)\mathbf{W} \mathbf{D}(\graph)^T$, is similar to
    \begin{align}
        \begin{bmatrix}
        \lap_{e,s} \mathbf{RWR}^T & {\mathbf{0}}\\ {\mathbf{0}} & {\mathbf{0}}
        \end{bmatrix},
    \end{align}
    where $\lap_{e,s} = \mathbf{D}(\graph_\tau)^T \mathbf{E}^{-1} \mathbf{D}(\graph_\tau)$is the edge Laplacian for a spanning tree $\graph_\tau$ which is symmetrically ``weighted'' by the time scaling parameters, and
$
\mathbf{R}(\graph) = \begin{bmatrix} I & \mathbf{T}^{c}_\tau \end{bmatrix} = R\otimes I_k,
$
\noindent where $R$ is the basis of the cut space of $\graph$ as defined as in~\cite{Zelazo2011a}, with
$
\mathbf{T}^c_\tau = (\mathbf{D}(\graph_\tau)^T \mathbf{D}(\graph_\tau))^{-1} \mathbf{D}(\graph_\tau)^T \mathbf{D}(\graph_c).
$
\noindent Here, the $\tau$ and $c$ subscripts on $\graph$ denote the incidence matrices for a spanning tree and the complementary edges in $\graph$, respectively. 

\end{theorem}
\begin{proof}
First, we use a lemma, proven in the Appendix.
\begin{lemma}
  \label{lem:1}
  The following hold:
$
    \mathbf{T}_\tau^c = T_\tau^c \otimes I_k,~
    \mathbf{D} = \mathbf{D}_\tau \mathbf{R}.
$
\end{lemma}

\medskip

Following \cite{Zelazo2011a,Foight2019a}, we define a similarity transformation,
\begin{align}
    S_v(\graph) &= \begin{bmatrix}
        \E^{-1}\D(\graph_\tau)\left( \D(\graph_\tau) ^T \E^{-1} \D(\graph_\tau)\right)^{-1} & {\mathbb{1}}
    \end{bmatrix}\\
    S_v(\graph)^{-1} &= \begin{bmatrix}
        \D(\graph_\tau)^T \\ \Xi^{-1} F
    \end{bmatrix},\\
      \mathbb{1}&= \mathbf{1}_n \otimes I_k,~
      F = 
          \begin{bmatrix}
            E_1 & \cdots & E_n
          \end{bmatrix}\\
                  \Xi &= \mathrm{diag}\left(\{\epsilon_{s,i}\}_{i=1}^k\right),~
      \epsilon_{s,i} = \sum_{j=1}^n \epsilon_{j,i}.
    \end{align}
    Note that $\epsilon_{s,i}$ is the sum of all the time scale parameters of the $i$-th substate over all nodes.
    We establish that this transformation is well-defined in the Appendix.
    \begin{lemma}
      \label{lem:2}
      The similarity transforms are well-defined, in that $S_v^{-1}S_v = I$.
    \end{lemma}
    
    Next, denoting for brevity $\D_\tau:= \D(\graph_\tau)$, $\D:=\D(\graph)$, etc, and noting that $\mathbf{D}^T \mathbb{1} = (D^T\mathbf{1})\otimes I_k = \mathbf{0}$, we conclude,
    \begin{align}
        &S_v^{-1} \lap_{w,s}(\graph) S_v= \\
      & \begin{bmatrix}
           \D_\tau^T \E^{-1} \mathbf{D W D}^T \E^{-1} \D_\tau \left( \D_\tau \E^{-1} \D_\tau\right)^{-1} & A\\
           \Xi^{-1} F \E^{-1}\D\W\D^T \E^{-1}\D_\tau \left(\D_\tau \E^{-1} \D_\tau\right)^{-1} & B
       \end{bmatrix}\\
&(\text{where }A= \D_\tau^T \E^{-1} \mathbf{D W D}^T \mathbb{1},~B=\Xi^{-1} F \E^{-1} \D\D^T\mathbb{1})\\
 &=        \begin{bmatrix}
   \D_\tau^T E^{-1} \D_\tau \mathbf{R W R}^T \D_\tau^T \E^{-1} \D_\tau \left( \D_\tau \E^{-1} \D_\tau\right)^{-1} & \mathbf{0}\\
   \Xi^{-1} \mathbb{1}^T\D_\tau \R\W\R^T & \mathbf{0}
        \end{bmatrix}\\
&=         \begin{bmatrix}
        \lap_{e,s} \mathbf{R W R}^T & \mathbf{0}\\ \mathbf{0} & \mathbf{0}
        \end{bmatrix}.
    \end{align}
\end{proof}

By noting that $S_v \mathbf{x}_e(t) = \mathbf{x}(t)$, the scaled, matrix weighted consensus model with noise~\eqref{eq:consensus} is equivalent to,
\begin{equation} \label{eq:edgemodel}
    \begin{aligned}
    \dot{\mathbf{x}}_e(t) & = \begin{bmatrix} -\lap_{e,s}(\graph_\tau) \mathbf{R}(\graph) \mathbf{W} \mathbf{R}(\graph)^T & \mathbf{0} \\ \mathbf{0} & \mathbf{0} \end{bmatrix} \mathbf{x}_e(t)  \\
    & \quad + \begin{bmatrix} \D_{\tau}^T \E^{-1} & -\lap_{e,s}(\graph_\tau) \mathbf{R}(\graph) \\ \Xi^{-1} \mathbb{1}^T & {\mathbf{0}} \end{bmatrix} \begin{bmatrix} \omega(t) \\ \mathbf{v}(t) \end{bmatrix} \\
    \mathbf{z}(t) & = \begin{bmatrix} \R(\graph) & {\mathbf{0}} \end{bmatrix} \mathbf{x}_e(t).
    \end{aligned}
\end{equation}
\noindent We can see that the form of~\eqref{eq:edgemodel} naturally suggests a partitioning of the edge state variable into a set of states in the spanning tree and those in the consensus space (span$(\mathbf{1}\otimes I_k)$), $\mathbf{x}_e(t) = \begin{bmatrix} \mathbf{x}_\tau(t) & \mathbf{x}_{\mathbf{1}}(t) \end{bmatrix}$. The resulting dynamics for the spanning tree states is taken from~\eqref{eq:edgemodel} as,
\begin{equation} \label{eq:treemodel}
\Sigma_\tau := \left\{
    \begin{aligned}
    \dot{\mathbf{x}}_\tau (t) & = -\lap_{e,s}(\graph_\tau) \R(\graph) \mathbf{W} \R(\graph)^T \mathbf{x}_\tau (t) \\
    & \quad + \D_{\tau}^T \E^{-1} \Omega \hat{\omega} - \lap_{e,s}(\graph_\tau) \R(\graph) \Gamma \hat{\mathbf{v}} \\
    \mathbf{z}(t) & = \R(\graph)^T \mathbf{x}_\tau(t),
    \end{aligned} \right.
\end{equation}
\noindent where $\hat{\mathbf{v}}$ and $\hat{\omega}$ are normalized noise signals, $\Omega = \mathbb{E}\left[\omega(t) \omega(t)^T\right]$, and $\Gamma = \mathbb{E}\left[\mathbf{v}(t) \mathbf{v}(t)^T\right]$. An important note is that the chosen output for~\eqref{eq:model} results in the output of~\eqref{eq:treemodel} containing information of the cycle states due to the fact that the cycle states are linear combinations of the tree states. 




The $\mathcal{H}_2$ performance of~\eqref{eq:treemodel} is given by $\mathbf{tr}(R^T X^\star R)$, where $X^\star$ is the positive-definite solution to the Lyapunov equation,
\begin{align}
& -\lap_{e,s}^\tau \R \W \R^T X - X \R \W \R^T \lap_{e,s} ^\tau + \D_{\tau}^T \E^{-1} \Omega \Omega^T \E^{-1} \D_{\tau} + \nonumber \\
& \quad \lap_{e,s}^\tau \R \Gamma \Gamma^T \R^T \lap_{e,s}^\tau = 0. \label{eq:lyapunov}
\end{align}
In general, the addition of the matrix weighting and scaling precludes a closed form solution to~\eqref{eq:lyapunov} (which is desirable to find $X$'s dependence on $\E,\W$), and numeric results yield a nontrivial mixing of weights and scaling parameters in the entries of $X$. However, in the following section we will outline a case when analytic solutions to~\eqref{eq:lyapunov} exist, providing insights for design of edge weights and scaling parameters for optimal performance.


\section{$\Htwo$ Performance and Design Problems}

\subsection{$\mathcal{H}_2$ Performance} \label{sec:h2-results}

In this section, we discuss the $\mathcal{H}_2$ performance for the models of edge consensus in the cases of nodes with time scales, and matrix weighted edges. Specifically, we identify an interesting case where explicit solutions to~\eqref{eq:lyapunov} can be found by an appropriate choice of noise covariances, and then discuss how the general solution can be approximated by this choice.
%

By inspection of~\eqref{eq:lyapunov}, we can note that by selecting the covariance matrices $\Omega = \sigma_\omega \E^{1/2}$ and $\Gamma = \sigma_v \W^{1/2}$, we can find an analytic solution for~\eqref{eq:lyapunov} in line with those in~\cite{Zelazo2011a,Foight2019a}.
With this choice,~\eqref{eq:lyapunov} has the solution,
\begin{equation} \label{eq:lysol}
X^\star = \frac{1}{2}\left(\sigma_w^2 (\R \W \R^T)^{-1} + \sigma_v^2 \lap_{e,s}^\tau \right).
\end{equation}
This solution is of particular interest because the edge and node weightings are separated in their effect on the $\mathcal{H}_2$ performance, save for the placement of the $\sigma_{\omega}$ and $\sigma_v$ parameters (that is, the effective covariance parameter of the process noise is a ``node'' parameter, but multiplies the term containing the edge weighting in~\eqref{eq:lysol}, and vice versa). Note that while this choice of $\Omega$ and $\Gamma$ allows for the analytic solution~\eqref{eq:lysol}, this solution is merely a proxy for the true performance of the system; for it to be useful for analysis, the error induced by the choice of covariances needs to be assessed.

Consider the scenario in which we have some given covariance matrices, $\Omega_T$ and $\Gamma_T$, which when used to solve~\eqref{eq:lyapunov} yield the ``true'' performance of the system.
These matrices may or may not be known to us; in either case a relevant endeavor is the quantification of the error incurred by estimating the true performance by~\eqref{eq:lysol}.
The method we will use to quantify this error depends on the positive-definite ordering of performances given by covariance matrices.
The following lemma, proven in the Appendix, provides this ordering. 

\begin{lemma}[Ordering of Performance] \label{lem:ordering}
For covariance matrices satisfying,
$
\underline{\Omega} \preceq \Omega \preceq \bar{\Omega}$ and $
\underline{\Gamma} \preceq \Gamma \preceq \bar{\Gamma},
$
\noindent then denoting the solution to~\eqref{eq:lyapunov} using $(\underline{\Omega},\underline{\Gamma});\ (\Omega,\Gamma);\ (\bar{\Omega},\bar{\Gamma})$ as $\underline{X};\ X;\ \bar{X}$, we have, $\underline{X} \preceq X \preceq \bar{X}$.
From this we can conclude,
\[
\mathbf{tr}(R^T \underline{X} R) \leq \mathbf{tr}(R^T X R) \leq \mathbf{tr}(R^T \bar{X} R).
\]
\end{lemma}

We can employ Lemma~\ref{lem:ordering} to place a bound on the potential error of making the assumption that gives~\eqref{eq:lysol}. Observe that it is possible to choose multiplicative factors such that,
\begin{align*}
\underline{\alpha} \E^{1/2} & \preceq \Omega_T \preceq \bar{\alpha} \E^{1/2} \\
\underline{\beta} \W^{1/2} & \preceq \Gamma_T \preceq \bar{\beta} \W^{1/2}.
\end{align*}
From Lemma~\ref{lem:ordering}, we know that the true performance will lie within the performances calculated using $(\underline{\alpha} \E^{1/2},\underline{\beta} \W^{1/2})$ and $(\bar{\alpha} \E^{1/2},\bar{\beta} \W^{1/2})$. Thus, taking the difference between the maximum performance and the minimal performance gives the worst case error of~\eqref{eq:lysol}, which is given by,
\begin{equation} \label{eq:perf_bound}
\begin{aligned}
\mathbf{tr}(R^T \bar{X}^{\ast} R) - \mathbf{tr}(R^T \underline{X}^{\ast} R) & = \frac{1}{2} (\bar{\alpha} - \underline{\alpha})\sigma_w^2 \mathbf{tr}((\R \W \R^T)^{-1}) \\
& + \frac{1}{2} (\bar{\beta} - \underline{\beta})\sigma_v^2 \mathbf{tr}(\lap_{e,s}^\tau).
\end{aligned}
\end{equation}
From this result, we can see that the relative error will be determined by the relative sizes of the multiplicative factors, which raises the question of how the factors can be found or chosen. From the definition of the necessary ordering, it is of course sufficient that,
\begin{equation} \label{eq:sufficent}
\begin{aligned}
\bar{\alpha} & = \frac{\lambda_{\max}(\Omega_T)}{\epsilon_{\max}^{1/2}};\ \underline{\alpha} = \frac{\lambda_{\min}(\Omega_T)}{\epsilon_{\min}^{1/2}} \\
\bar{\beta} & = \frac{\lambda_{\max}(\Gamma_T)}{\lambda_{\max}(\W)^{1/2}};\ \underline{\beta} = \frac{\lambda_{\min}(\Gamma_T)}{\lambda_{\min}(\W)^{1/2}}.
\end{aligned}
\end{equation}
The sufficiency, as opposed to necessity, of the parameters given in~\eqref{eq:sufficent} results in conservative performance bounds. The bounds can be tightened by solving for the minimal/maximal parameters via a simple (convex) optimization problem of the form (for $\bar{\alpha}$),
\begin{equation} \label{eq:optparam}
\begin{aligned}
\min_{\alpha}\ &\ \alpha \\
\text{s.t.}\ &\ \Omega_T - \alpha E^{1/2} \preceq 0.
\end{aligned}
\end{equation}
\noindent The above problem can be modified to give the minimal/maximal values of $\bar{\beta}$ and $\underline{\alpha}/\underline{\beta}$. While~\eqref{eq:optparam} gives an obvious advantage over~\eqref{eq:sufficent}, it does require complete information about the covariances, whereas~\eqref{eq:sufficent} requires only the spectral bounds of the true covariances. In either case, however, it can be seen that the worst case error is helped when the true covariances and the assumed covariances have similar spectral bounds, that is, if the minimum and maximum eigenvalues of $\Omega_T/\Gamma_T$ are approximately equal to those of $\sigma_w E^{1/2}/\sigma_v W^{1/2}$, the necessary multiplicative factors will be approximately unity and the possible discrepancy small (within a factor of $\simeq$5-10). We can see this numerically for random graphs on $n=20$ nodes with $k=2$ in Figure~\ref{fig:perfbnd}, where the true covariances, edge weights, and node time scales were randomly generated then scaled to align maximum/minimum eigenvalues. Numerical results over a range of $n$ and $k$ suggest that as the number of substates increases, the bounds (calculated via~\eqref{eq:optparam}) become more conservative.

\begin{figure}
  \centering
  \includegraphics[width=0.98\columnwidth]{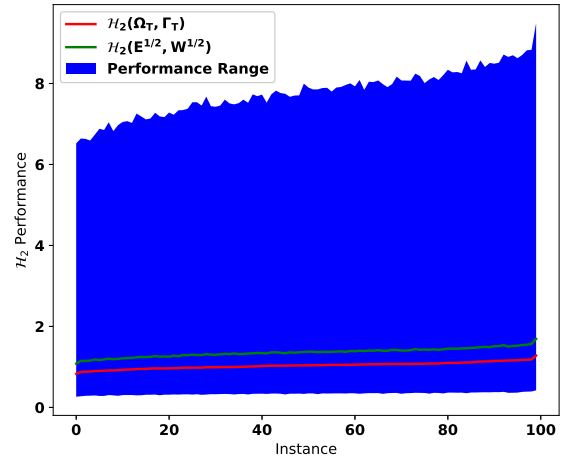}
  \caption{Numerically calculated $\Htwo$ performance for random graphs ($n=20,k=2$) with random covariances ($\Htwo(\Omega_T,\Gamma_T)$) and the solution to~\eqref{eq:lysol} ($\Htwo(\E^{1/2},\W^{1/2})$). The blue shaded region represents the possible range in performance, calculated using the parameters found via optimization problems given by~\eqref{eq:optparam}.}
  \vspace{-5mm}
  \label{fig:perfbnd}
\end{figure}
While the results of Lemma~\ref{lem:ordering} (along with~\eqref{eq:sufficent} or~\eqref{eq:optparam}) give the ability to assess whether or not~\eqref{eq:lysol} is an acceptable proxy for the true performance, there are no explicit guarantees that the bounds are tight enough for all applications. In cases where there is large discrepancy between the spectral bounds of the true covariances and the time scale or edge weight matrices, for example,~\eqref{eq:lysol} may not be useful for performance estimation or optimization. For the cases where it is an appropriate proxy, however, we proceed with time scale and edge weight design to promote minimal $\Htwo$ performance, starting with a remark pertaining to the special case of tree graphs.
%
%
\begin{remark}

When the underlying graph topology is a tree, $\R = I$, and~\eqref{eq:lysol} simplifies to,

\begin{equation*}
X^\star = \frac{1}{2}\left(\sigma_{\omega}^2 \W^{-1} + \sigma_v^2 \lap_{e,s}^\tau \right).
\end{equation*}

\noindent Furthermore, in this case $\mathcal{H}_2(\Sigma_\tau) = \mathbf{tr}(X^\star)$. A closed form solution for the performance in this case is given in the following lemma. 
\end{remark}

\begin{lemma} For a tree graph, the $\Htwo$ norm of the $\Sigma_{\tau}$~\eqref{eq:treemodel} system is given by,
\begin{align}
\mathcal{H}_2(\Sigma_\tau) & = \frac{1}{2} \mathbf{tr}\left(\sigma_{\omega}^2 \W^{-1} + \sigma_v^2 \lap_{e,s}^\tau \right) \nonumber \\
& = \frac{1}{2} \left( \sigma_{\omega}^2 \sum_{j=1}^{n-1} \mathbf{tr}\left(W_{j}^{-1}\right) + \sigma_v^2 \sum_{l=1}^{k} \sum_{i = 1}^{n} \frac{\text{deg}(\nu_i)}{\epsilon_{i,l}} \right), \label{eq:treeresult}
\end{align}
\noindent where deg$(\nu_i)$ is the unweighted degree of agent $\nu_i$, and $j$ is the index over the edges. 
\end{lemma}
\begin{proof}
First consider the weight term. $\W$ is block-diagonal, so $\W^{-1} = \mathrm{Blkdiag}(W_1^{-1},\dots,W_{n-1}^{-1})$. Thus, its trace is the sum of the traces of the edge weight matrix inverses. Now, consider a single layer of substates, denoted by $l$. For the second term, consider one of the diagonal elements of $\lap_{e,s}$,
\begin{equation*}
[L_{e,s}^\tau]_{(ql)(ql)} = a_q^T E_l^{-1} a_q = \epsilon_{i,l}^{-1} + \epsilon_{j,l}^{-1},
\end{equation*}
\noindent where $a_q$ is the edge vector corresponding to the edge between nodes $i$ and $j$, that is, $q = \kappa(ij)$. Now consider a single node, $\nu_i$. In the sum over all edges of the graph, $\epsilon_{i,l}^{-1}$ will appear once for every edge that connects $\nu_i$ to its neighbors, which is the unweighted degree of $\nu_i$. Considering all other nodes gives the result for the second term. Finally, the preceding argument holds for all the sub-state layers, which gives the sum over all sub-states. 
\end{proof}







From this result, we can see that there exists a trade off between the time scale parameters and the topology (in this case, the degree distribution) which determines the overall performance of the network. Also, we can contrast the influence of time scale parameters and edge weights in this case. For a given distribution of scaling parameters and edge weights, changing the assignment of edge weights does not affect the $\mathcal{H}_2$ performance contribution from a given sub-state layer. However, the assignment of scaling parameters can have a significant effect on the performance of the network, which is in line with the similar results in the context of single-input influenced consensus~\cite{Foight2019}. 

We can also see that the performance contribution from the time scale parameters in the matrix weighted case is identical to considering a network with $k$ disconnected layers, where each layer has its own time scale distribution. Thus, the evaluation of the time scale assignment in the matrix weighted case is effectively the same as considering assignment in the scalar case. With this in mind, we direct interested readers to~\cite{Foight2019a} for an example of this assignment in action.

\subsection{Gradient Updates on Edge Weights} \label{sec:gu-design}

In the previous section, we saw that one could separate the contributions of the time scales and the edge weights on the $\mathcal{H}_2$ norm.
We now present a design problem for optimizing the edge weight term of Equation~\eqref{eq:lysol}.
Consider Problem~\eqref{eq:update2},
\begin{align}
\begin{aligned}
\min_{\{W_i\}_{i=1}^{|\Edges|}} & \mathbf{tr}\left(\R^T (\R\W\R^T)^{-1} \R\right) + \frac{h}{2}\sum_{e\in\Edges}\mathbf{tr}\left[\W_e^T\W_e\right]^2 \\
\text{s.t.}\ & {W_{\min}} \preceq {W_e} \preceq {W_{\max}}, \forall e \in \Edges \\
&\ W = \mathrm{blkdiag}(W_i).
\end{aligned} \label{eq:update2} \tag{P1} 
\end{align}
We include a regularization term in the cost function to avoid the trivial solution of completely disconnecting the graph, as well as upper/lower bounds on the matrix weights.
A gradient update for solving Problem~\eqref{eq:update2} is derived in Proposition~\ref{prop:1}, and proven in the Appendix.

\begin{proposition}[Gradient Update for Edge Weights]
\label{prop:1}
The gradient of the cost function with respect to the edge weight $W_e$ in Problem~\eqref{eq:update2} is given by 
\begin{align}
  \nabla_{W_e} f[H] = - \mathrm{deblk}_e^k\left[Q^TQ\right] + h W_e,
\end{align}
where $\mathrm{deblk}_e^k[Q^TQ]$ is the $e$th $k\times k$ diagonal block of $Q^TQ$, and $Q$ is given by 
\begin{align}
  Q &= \R^T \left( \R \W_H^c \R^T \right)^{-1} \R\\
  \W_H^c&= \mathrm{blk}_c^k (H) + \sum_{l\in\Edges\setminus c} \mathrm{blk}_l^k(W_l),
\end{align}
where
\begin{align}
  \mathrm{blk}_c^k(H) =  \mathbf{e}_c  H \mathbf{e}_c^T = \left( e_c\otimes I_k\right) H \left( e_c^T \otimes I_k\right)
\end{align}
denotes the $kn\times kn$ matrix with $H$ on the $c$th $k\times k$ diagonal block, with zeros otherwise.

A gradient update scheme for solving Problem~\eqref{eq:update2} is therefore 
\begin{align}
  W_e^{k+1} &= W_e^k - \dfrac{1}{h\sqrt{k}} \nabla_{W_e} f[H]\\
            &= W_e^k - \dfrac{1}{h\sqrt{k}} \left(h W_e - \mathrm{deblk}_e^k\left[Q^TQ\right] \right).
\end{align}
\end{proposition}

\subsection{Decentralized Time Scale Assignment} \label{sec:ts-design}

We saw previously in the definition of~\eqref{eq:update2} that a regularization term was included to prevent the trivial solution of disconnecting the graph. In the optimization of the time scale term of Equation~\eqref{eq:lysol}, this trivial solution takes the form of all agents/substates adopting the slowest time scale parameter possible. Thus, consider~\eqref{eq:update},
\begin{align}
\begin{aligned}
\min_{\epsilon_{1,1}^{-1},\dots,\epsilon_{n,k}^{-1}}\ &\ \frac{1}{2} \mathbf{tr}\left(\R^T \lap_{e,s}^\tau  \R\right) + \frac{h}{2}\sum_{i=1}^n \sum_{j=1}^k \epsilon_{i,j}^r \\
\text{s.t.}\ &\ {\epsilon}^{-1}_{\max} \leq {\epsilon_{i,j}}^{-1} \leq {\epsilon}^{-1}_{\min}\ \forall i \in \Nodes,\ j \in [k].
\end{aligned} \label{eq:update} \tag{P2} 
\end{align}
This is a minimization of the time scale portion of the separated $\mathcal{H}_2$ performance. A regularization term $2^{-1}h \sum_{i=1}^n \sum_{j=1}^k \epsilon_{i,j}^r$ penalizes large time scales for all nodes and their substates assuming positive, integer $r$. In the following proposition (proven in the Appendix), we show an analytic solution for the optimal time scale assignment which minimizes the $\mathcal{H}_2$ performance. 

\begin{proposition}[Analytic Optimal Time Scale Assignment]
\label{prop:2}
Consider~\eqref{eq:update}. Let the region defined by the box constraints on $1/\epsilon_{i,j}$ be denoted by $\mathcal{C}$. Then, the minimizing assignment of time scale parameters is given by,
\[
\epsilon_{i,j}^\ast = \text{Proj}_\mathcal{C}\left[\left(\frac{\text{deg}(\nu_i)}{h r}\right)^{\frac{1}{r+1}} \right].
\]
\end{proposition}

\begin{remark}
The assignment rule in Proposition~\ref{prop:2} is decentralized, as the optimal assignment value depends only on the (unweighted) degree of the $i$-th node and the parameters $h$ and $r$, which are locally known to the $i$-th node without global knowledge of the network topology.
\end{remark}

From this result we can see that for a class of regularization terms, the optimal time scale assignment is again driven by the degree distribution, which is in-line with the previous results. It is conceivable to consider using this result with online signal identification to locally adjust time scales in response to adversarial noise entering the system.


\section{Example: Flocking via Second-Order Consensus}
\label{sec:example}

Flocking is a behaviour exhibited by certain multi-agent systems that are coordinating their motion into a cohesive formation, for example birds or stampeding buffalo.
A consensus-type algorithm can be proposed that allows a system of $n$ agents to agree on their velocity vector while maintaining a separation from their neighbours~\cite{Chapman2015}.

\subsection{Matrix-Valued Double Integrator Consensus}

In the case of vector-valued states, we can write the dynamics as 
\begin{align}
 \begin{bmatrix}
    \dot{\mathbf{x}} \\ \E\ddot{\mathbf{x}}
  \end{bmatrix} = 
  \begin{bmatrix}
    \mathbf{0} & I \\
    - \lap_w & - \lap_w
  \end{bmatrix}
                 \begin{bmatrix}
                   \mathbf{x}\\\dot{\mathbf{x}}
                 \end{bmatrix}+
  \begin{bmatrix}
    \mathbf{0} & \mathbf{0}\\
    I & - \D_\graph
  \end{bmatrix}
        \begin{bmatrix}
          \omega \\ v
        \end{bmatrix}.
\end{align}
\begin{theorem}
  The double-integrator edge consensus model is given by 
  \begin{align}
  \begin{bmatrix}
    \dot{\mathbf{x}}_e\\ \ddot{\mathbf{x}}_e
  \end{bmatrix} &= 
  \begin{bmatrix}
    \mathbf{0} & 0 & ~~~~I & \\
      -\lap_e \R\W\R^T & \mathbf{0} & -\lap_e \R\W\R^T & \mathbf{0}\\
      \mathbf{0} & 0 & \mathbf{0} & 0
  \end{bmatrix}
\begin{bmatrix}
  \mathbf{x}_e \\ \dot{\mathbf{x}}_e
\end{bmatrix}\\
&+
  \begin{bmatrix}
    \mathbf{0} & \mathbf{0}\\
    \D_\tau^T\E^{-1} & -\lap_e \R\\
    \Xi^{-1} \mathbb{1}^T & \mathbf{0}
  \end{bmatrix}
        \begin{bmatrix}
          \omega \\ v
        \end{bmatrix},    
  \end{align}
and so the double-integrator consensus on the edge states of the chosen spanning tree $\tau$ is given by 
\begin{align}\label{eq:6}
  \begin{split}
  \begin{bmatrix}
    \dot{\mathbf{x}}_\tau \\ \ddot{\mathbf{x}}_\tau
  \end{bmatrix}&= 
                 \begin{bmatrix}
                   \mathbf{0} & I \\
                   -\lap_e \R\W\R^T & -\lap_e \R\W\R^T
                 \end{bmatrix}
                                   \begin{bmatrix}
                                     \mathbf{x}_\tau \\ \dot{\mathbf{x}}_\tau
                                   \end{bmatrix}\\
&+
  \begin{bmatrix}
    \mathbf{0} & \mathbf{0}\\
    \D_\tau^T & -\lap_e \R
  \end{bmatrix}
               \begin{bmatrix}
                 \omega \\ v
               \end{bmatrix}.
               \end{split}
\end{align}
\end{theorem}
\begin{proof}
Appling the coordinate transform $\mathbf{x}_e = S_v \mathbf{x}$ and following a similar calculation as in Theorem~\ref{thm:similarity} yields the result.
\end{proof}

We can also explicitly compute the form of the $\mathcal{H}_2$ norm for the matrix-weighted double-integrator consensus.
\begin{theorem}
  \label{thr:2}
  The controllability gramian for the time scaled double-integrator consensus is given by 
  \begin{align}
    \X^* &= \dfrac{1}{2} 
    \begin{bmatrix}
      \X_1& \mathbf{0}\\
      \mathbf{0}& \X_2
    \end{bmatrix}\\
   \X_1 &=\sigma_w^2 \left[\left(\R\W\R^T\right)^{-1}\lap_e^{-1}\left(\R\W\R^T\right)^{-1} \right] +\sigma^2_v\left(\R\W\R^T\right)^{-1}\\
    \X_2 &= \sigma_w^2\left(\R\W\R^T\right)^{-1} + \sigma_v^2 \lap_e.
  \end{align}
  Furthermore, the blocks $\X_1,\X_2$ of $\X^*$ correspond to the position and velocity states, respectively.
  Hence, one can consider the $\Htwo$ performance of the position and velocity states separately or aggregately by examining the $\Htwo$ norms 
  \begin{align}
    \Htwo^{(1)} = \frac{1}{2}\mathbf{tr}&\left(\R^T\X_1\R\right),~\Htwo^{(2)} = \frac{1}{2}\mathbf{tr}\left(\R^T\X_2\R\right)\\\Htwo^{(1)} &= \frac{1}{2}\mathbf{tr}\left(\R^T(\X_1+ \X_2)\R\right).
  \end{align}
\end{theorem}
\begin{proof}
  From the dynamics~\eqref{eq:6}, the controllablity gramian is given by the positive semi-definite solution to the Lyapunov equation 
  \begin{align}
    A\X^* + \X^*A^T + BB^T = \mathbf{0},\label{eq:7}
  \end{align}
  where $\X^*,A,$ and $BB^T$ are given by,
  \begin{align}
    A &= 
        \begin{bmatrix}
          \mathbf{0} & I \\
          -\lap_e \R\W\R^T & -\lap_e \R\W\R^T
        \end{bmatrix},\\
    BB^T &= 
           \begin{bmatrix}
             \mathbf{0}&\mathbf{0}\\
             \mathbf{0}&\sigma^2_w \lap_e + \sigma^2_v\lap_e \R\R^T \lap_e
           \end{bmatrix},\\
    \X^* &= 
          \begin{bmatrix}
            \X_1 & \X_3 \\ \X_3& \X_2
          \end{bmatrix}.
  \end{align}
  Solving Equation~\eqref{eq:7} yields 
  \begin{align}
    \X_3 &= \mathbf{0},\ \X_2 = \sigma_w^2\left(\R\W\R^T\right)^{-1} + \sigma_v^2 \lap_e\\
    \X_2 &= \X_1\R\W\R^T\lap_e,
  \end{align}
  and solving for $\X_1$ in the last display yields 
  \begin{align}
    \resizebox{\columnwidth}{!}
    {
    $\X_1 = \sigma_w^2 \left[\left(\R\W\R^T\right)^{-1}\lap_e^{-1}\left(\R\W\R^T\right)^{-1} \right] +\sigma^2_v\left(\R\W\R^T\right)^{-1}.$
    }
  \end{align}%
To measure the position, velocity, or both states for consideration in the $\Htwo$ norm, one can choose the observation matrices 
  \begin{align}
    \begin{bmatrix}
      \R^T & \mathbf{0}
    \end{bmatrix},~
    \begin{bmatrix}
       \mathbf{0} & \R^T
    \end{bmatrix},~
    \begin{bmatrix}
      \R^T & \mathbf{0}\\
      \mathbf{0} & \R^T
    \end{bmatrix},
  \end{align}
  respectively.
\end{proof}

\subsection{Numerical Example}

The weight update scheme applied to the second-order consensus problem was implemented numerically.
The task assigned to the agents was to use the second-order consensus protocol to achieve the formation shown in Figure~\ref{fig:formation} -- a formation assigned by sampling discrete points on a 2D spiral.

\begin{figure}
  \centering
  \includegraphics[width=0.6\columnwidth]{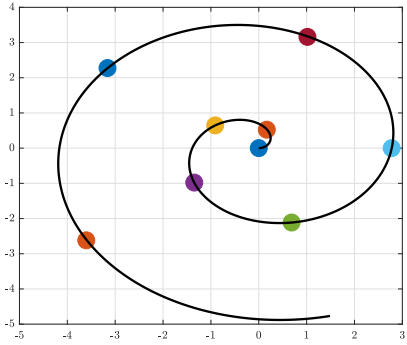}
  \caption{Relative formation of the agents, defined by discrete points on a spiral.}
  \label{fig:formation}
\end{figure}

Consensus on the formation is achieved by the second-order protocol with a constant signal $\mathbf{d}_i$ specifying the position in the formation:
\begin{align}
\resizebox{\columnwidth}{!}
{
  $\ddot{\mathbf{x}}_i = -\sum_{j\in \mathcal{N}(i)} \W_{ij}(\mathbf{x}_i- \mathbf{x}_j - \mathbf{d}_i) - \sum_{j \in \mathcal{N}(i)} \W_{ij}(\dot{\mathbf{x}}_i - \dot{\mathbf{x}_j}).$ 
}\label{eq:8}
\end{align}
\noindent Since this is a constant signal, the $\mathcal{H}_2$ performance of the edge states remains the same as in the previous section.

The initial selection of weights was chosen at random using the generator
\begin{align}
W_{\{i,j\}} = \alpha\left[\dfrac{G_{\{i,j\}} + G_{\{i,j\}}^T}{2} + 2 I_2\right],\label{eq:5}
\end{align}
where $G_{\{i,j\}}$ is a $2\times 2$ matrix with entries distributed according to a zero-mean standard Gaussian, and $\alpha=0.3$ was chosen arbitrarily to yield a suboptimal initial selection of weights.
The upper and lower bounds on the weights were chosen with the same generator in Equation~\eqref{eq:5}, but with $\alpha_l = 0.05$ and $\alpha_u = 10$. The initial time scale parameters were taken to be identically unity. The penalty parameter was chosen as $h=0.01$.

The gradient descent algorithm from Proposition~\ref{prop:1} converges quickly, and intermediate graph weights are visualized in Figure~\ref{fig:weights}.
None of the optimal edge weights saturated the upper and lower bounds in this setup. 
The minimizing time scale assignment was calculated using Proposition~\ref{prop:2}.


\begin{figure}
  \centering
  \includegraphics[width=0.9\columnwidth]{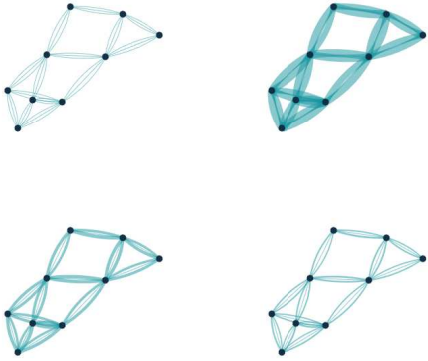}
  \vspace{-7mm}
  \caption{Visualization of edge weights over iterations $k=1$ (top left), 2 (top right), 3 (bottom left) \& 4 (bottom right). Each of the 3 independent parameters of the $2\times 2$ matrix-valued weight is visualized in a multigraph.}
  \label{fig:weights}
\end{figure}

These weights and scaling paramters were then used in a simulation of the dynamics in~\eqref{eq:8} over a time span of 30 seconds.
At $10\leq t\leq 20$, the formation is subject to a `gust' of noise on the nodes and edges with covariance $\sigma_w^2 I$ and $\sigma_v^2 \W$, with $\sigma_v=\sigma_w = 5$.
Simulations were performed with no updates (NUD), updates to edge weights (WUD) or time scales (TUD) during the wind gusts, and with both updates (BUD) during gusts.
The edge states for the $x$ and $y$ directions for these four cases are shown in Figure~\ref{fig:edgeStates}, and the variance away from the consensus value $x_e(t_f)$, $\mathrm{Var}[x_e(t)]:= \left[x_e(t) - x_e(t_f)\right]^2$, is shown in Figure~\ref{fig:edgeVar}.
The updated weights and scales outperform their initial, suboptimal values.

\begin{figure}
  \centering
  \includegraphics[width=0.9\columnwidth]{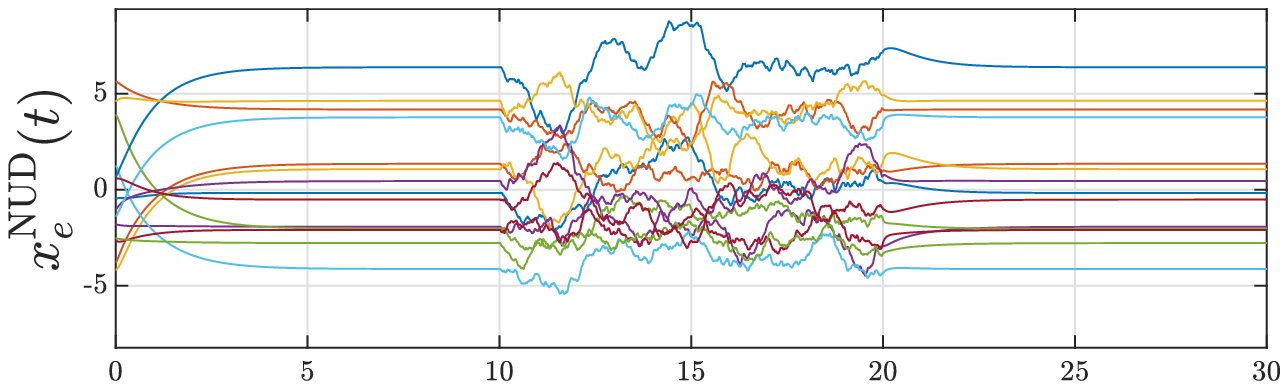}\\
  \includegraphics[width=0.9\columnwidth]{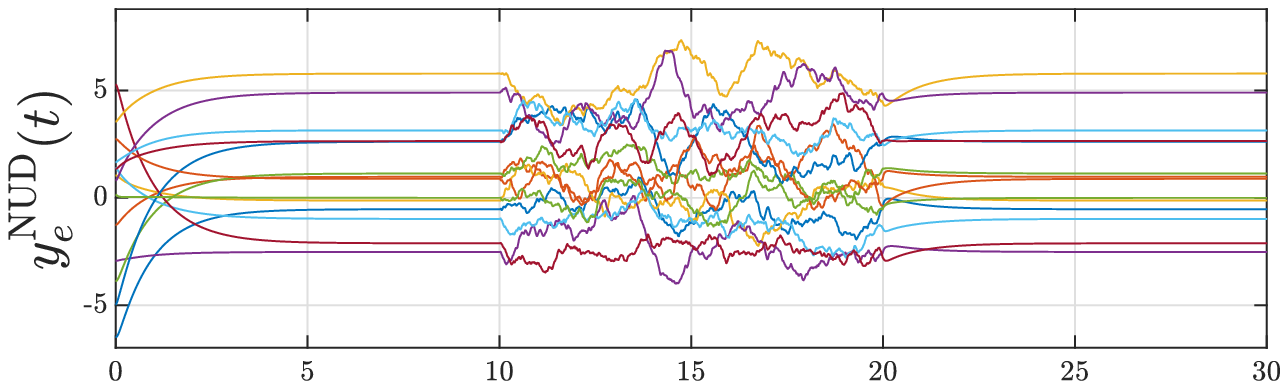}\\
  \includegraphics[width=0.9\columnwidth]{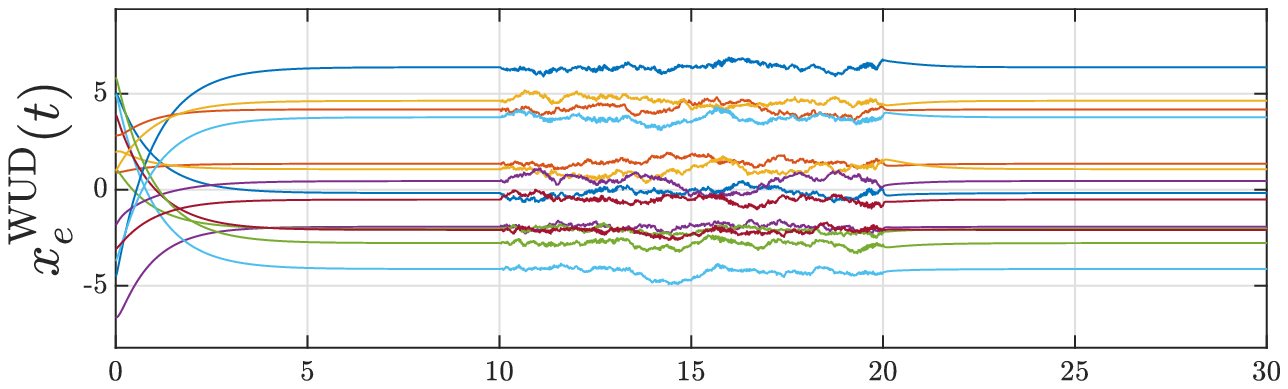}\\
  \includegraphics[width=0.9\columnwidth]{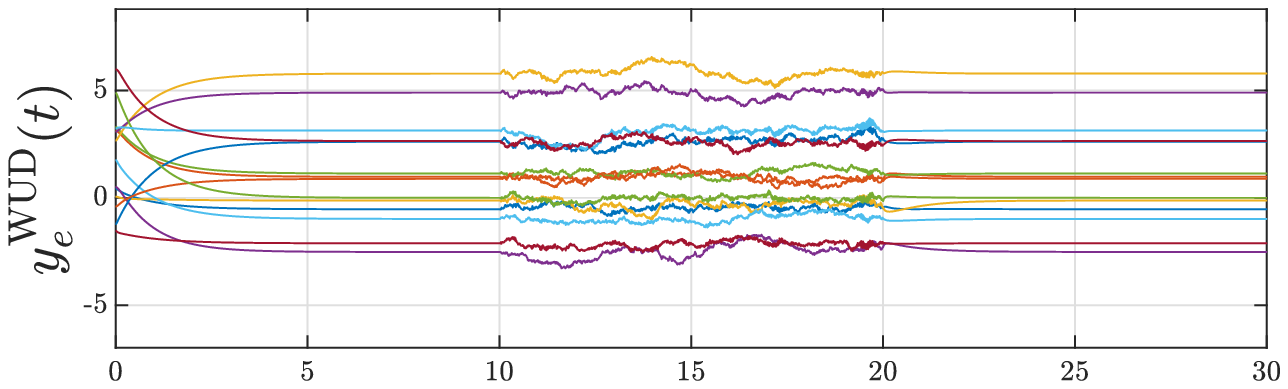}\\
  \includegraphics[width=0.9\columnwidth]{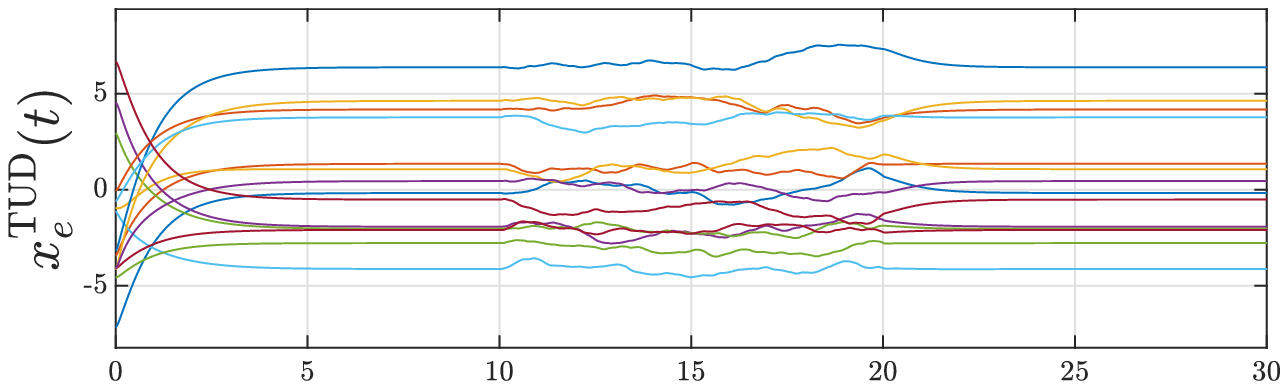}\\
  \includegraphics[width=0.9\columnwidth]{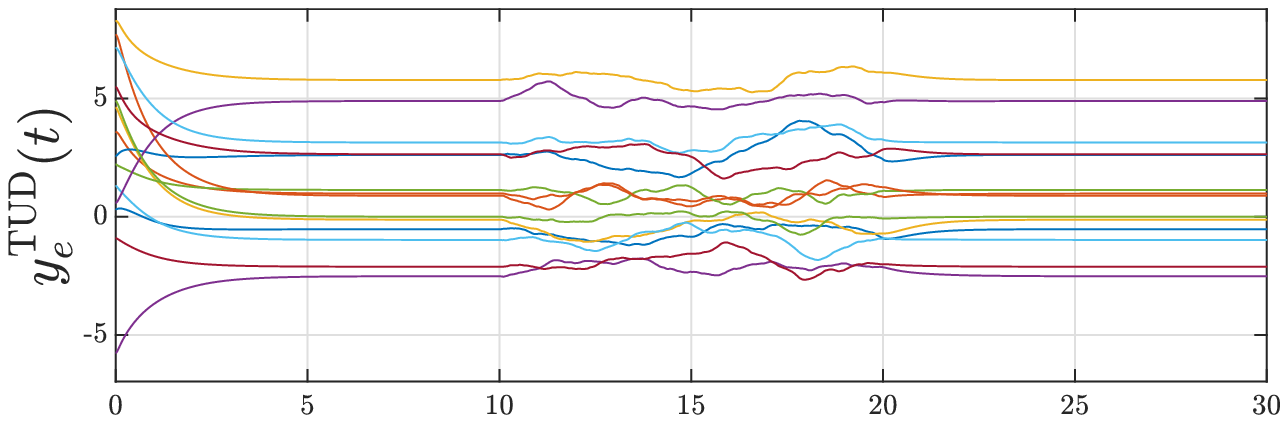}\\
  \includegraphics[width=0.9\columnwidth]{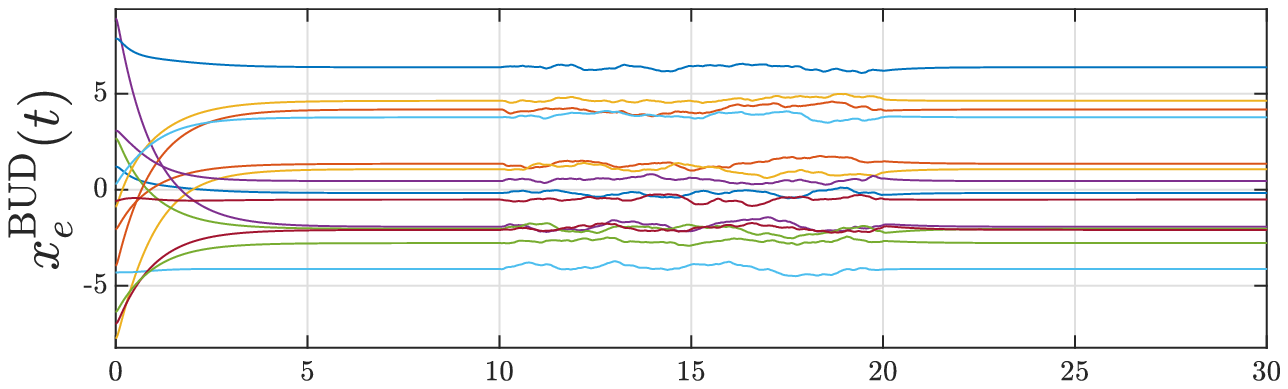}\\
  \includegraphics[width=0.9\columnwidth]{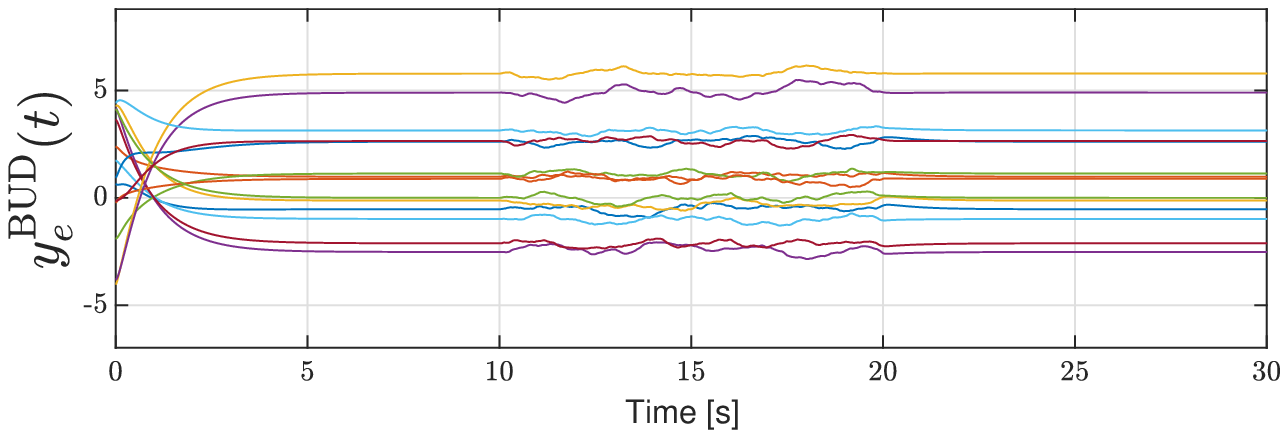}
  \caption{Edge states in $x,y$ directions over time subjected to gust at $10\leq t\leq 20$; cases from top to bottom are with no updates (NUD), weight updates (WUD), time scale update (TUD), and both updates (BUD). }
  \label{fig:edgeStates}
\end{figure}

\begin{figure}
  \centering
  \includegraphics[width=\columnwidth]{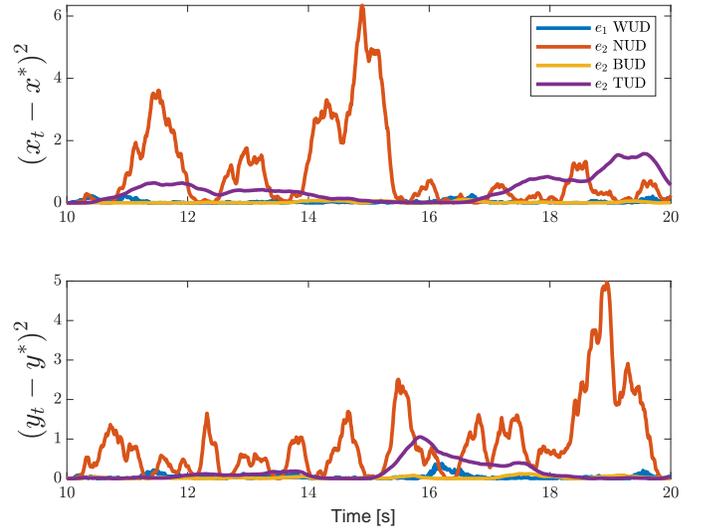}
  \caption{Variance from consensus for a representative edge $[x_e(t) - x_e(t_f)]^2$ for the case with the weight update (WUD), time scale update (TUD), both updates (BUD), and with no updates (NUD).}
  \label{fig:edgeVar}
\end{figure}


\section{Conclusion}
\label{sec:conclusion}

In this paper we have developed a framework for investigating noise-driven consensus on matrix-weighted and time scaled graphs, with the aim of minimizing the $\Htwo$ performance of such systems.
We identified a convenient choice of noise covariances that allows for the separation of the performance contributions from the edge weighting and time scales.
This allowed us to investigate the optimal assignment of time scale parameters and edge weights to promote network resilience.
Finally, we applied these results to a flocking example, where we observed that applying a time scale assignment and edge weight update in response to an injection of noise results in less perturbation of the agent states compared to the non-updated case, or optimizing over only the edges or only the time scales.

In real-world networked dynamical problems adversarial noise will likely not be applied to all nodes within the network as was taken to be the case in this work.
Thus, a useful extension of this work is to investigate what results can be found when the noise injection is limited to a subset of nodes, or in the matrix weighted case, a subset of substates.
Furthermore, interesting problems such as leader selection in noise driven consensus networks have previously been investigated, so these areas are also potential directions for future work by extending those related topics to the time scaled and matrix weighted case.


\section*{Acknowledgments}
This material is based upon work supported by the National Science Foundation Graduate Research Fellowship Program. Any opinions, findings, and conclusions or recommendations expressed in this material are those of the author(s) and do not necessarily reflect the views of the National Science Foundation.

\bibliographystyle{IEEEtran}
\bibliography{references}

\section*{Appendix} \label{sec:appendix}
Here, we prove Lemmas~\ref{lem:1}-\ref{lem:ordering}, and Propositions~\ref{prop:1} \& \ref{prop:2}.

\textbf{Lemma~\ref{lem:1}.}
\begin{proof}
  We can compute: 
  \begin{align}
    \T_\tau^c &= \left(\D_\tau ^T \D_\tau\right)^{-1} \D_\tau^T\D_c\\
              &=\left( \left( D_\tau^TD_\tau \right)^{-1} D_\tau D_c\right) \otimes I_k
              = T_\tau^C\otimes I_k,
  \end{align}
  and similarly,
  $\D = D\otimes I_k = \left(D_\tau \otimes I_k \right)\left( R\otimes I_k \right) = \D_\tau \mathbf{R}$.
\end{proof}

\textbf{Lemma~\ref{lem:2}.}
\begin{proof}
      Denote the product $S_v^{-1}S_v$ in block form: 
      \begin{align}
        S_v^{-1}S_v = 
        \begin{bmatrix}
          \mathcal{A}&\mathcal{B}\\\mathcal{C}&\mathcal{D}
        \end{bmatrix}.
      \end{align}
      First, note that 
      \begin{align}
        F\E^{-1} = 
        \begin{bmatrix}
          E_1&\cdots& E_n
        \end{bmatrix} \E^{-1} = 
                      \begin{bmatrix}
                        I&\cdots & I
                      \end{bmatrix}= \mathbf{1}^T_n \otimes I_k.
      \end{align}
      Then, we can compute each term: 
      \begin{align}
        \mathcal{A}&= \D_\tau ^T \E^{-1} \D_\tau \left(D_\tau^T \E^{-1} \D_\tau\right)^{-1} = I\\
        \mathcal{B}&= \D_\tau^T \mathbb{1} 
         = \left(D_\tau^T\mathbf{1}_n \right) \otimes I_k = 0 \otimes I_k = \mathbf{0}.\\
        \mathcal{C}&= \Xi^{-1} F \E^{-1} \D_\tau \left(\D_\tau^T \E^{-1} \D_\tau\right)^{-1}\\
         &= \Xi^{-1} \left(\mathbf{1}^T_n \D_\tau \otimes I_k\right) \left(\D_\tau^T \E^{-1} \D_\tau\right)^{-1}= \mathbf{0}.
      \end{align}
      Lastly, we have 
      \begin{align}
        \mathcal{D} &= \Xi^{-1} F \mathbb{1}\\
          &=
            \begin{bmatrix}
              \epsilon_{s,1}^{-1} & & \\
              & \ddots & \\
              & & \epsilon_{s,k}^{-1}
            \end{bmatrix}
                  \begin{bmatrix}
                    E_1 & \cdots & E_n
                  \end{bmatrix} \left( \mathbf{1}_n \otimes I_k\right)\\
          &=
            \begin{bmatrix}
              \epsilon_{s,1}^{-1} & & \\
              & \ddots & \\
              & & \epsilon_{s,k}^{-1}
            \end{bmatrix}
                  \begin{bmatrix}
                    \sum_{j=1}^n\epsilon_{j,1} & & \\
              & \ddots & \\
              & & \sum_{j=1}^n\epsilon_{j,k}
            \end{bmatrix}
= I_k.
      \end{align}
    \end{proof}

\textbf{Lemma~\ref{lem:ordering}.}
    \begin{proof}
For simplicity, we adopt the following notation the state and input matrix of~\eqref{eq:treemodel}, $A := -\lap_{e,s} \R \W \R^T$, and $B := B_{\tau} \Omega \Omega^T B_{\tau}^T + B_c \Gamma \Gamma^T B_c^T$. Due to the stability of $\lap(\graph)$, we know the solution to~\eqref{eq:lyapunov} is positive definite, and can be written as in integral form~\cite{skogestad2001},
\[
X = \int_0^\infty e^{A t} B_{\tau} \Omega \Omega^T B_{\tau}^T e^{A^T t} \diff  + \int_0^\infty e^{A t} B_c \Gamma \Gamma^T B_c^T e^{A^T t} \diff
\]
Consider either term in the above solution, and generalize as,
\[
X_p = \int_0^\infty e^{A t} B_p Z Z^T B_p^T e^{A^T t} \diff 
\]
\noindent where $Z$ and $B_p$ are placeholders for a covariance matrix and input matrix, respectively. From this, consider the quadratic form with any $u$,
\[
u^T X_p u = \int_0^\infty \| Z^T B_p^T e^{A^T t} u \|_2^2 \diff 
\]
\noindent This form is usually employed to show the positive definiteness of $X_p$ based on the stability of $A$, but here it can be used to order $X_p$'s based on the ordering of $Z$:
\begin{align*}
u^T(X_p - \bar{X}_p)u & = \int_0^\infty \| Z^T B_p^T e^{A^T t} u \|_2^2 - \| \bar{Z}^T B_p^T e^{A^T t} u \|_2^2 \diff  \\
& \leq \int_0^\infty  \left(\|Z^T\|_2^2 - \|\bar{Z}^T\|_2^2\right) \| B_p^T e^{A^T t} u \|_2^2 \diff 
\end{align*}
\noindent From $Z \preceq \bar{Z}$, it follows that $\|Z\|_2^2 \leq\|\bar{Z}\|_2^2$. Thus, this integral is always negative, which then implies $X_p - \bar{X}_p \preceq 0 \Rightarrow X_P \succeq \bar{X}_p$. By identical argument, 
\[
u^T(\underline{X}_p - X_p)u \leq  \int_0^\infty  \left(\|\underline{Z}^T\|_2^2 - \|Z^T\|_2^2\right) \| B_p^T e^{A^T t} u \|_2^2 \diff 
\]
\noindent implies that $\underline{X}_p \preceq X$. This shows the ordering for the solutions of~\eqref{eq:lyapunov} for ordered covariances. 

Now, consider the ordering of the performance, which is given by $\mathbf{tr}(R X R^T)$. From the ordering and positive definiteness of $\underline{X},X,\bar{X}$, we have, $u^T \underline{X} u \leq u^T X u \leq u^T \bar{X} u$, for all $u$. Letting $u = R x$ gives,
\[
x^T R^T \underline{X} R x \leq x^T R^T X R x \leq x^T R^T \bar{X} R x.
\]
Now, note that the trace of a matrix $M$ can be written as $\sum_i e_i^T M e_i$. Taking $M = R^T X R$ and $x = e_i$ in the above equation gives the desired ordering on the performances.
\end{proof}

\textbf{Proposition~\ref{prop:2}.}
\begin{proof}

We first identify the gradient of the cost function with respect to $W_c$, the weight on the $c$th edge in $\Edges$.
To this end, consider the functions 
\begin{align}
  f(X) &= \R \left[ \mathbf{e}_c X \mathbf{e}_c^T + \sum_{l \in \Edges\setminus c} \mathbf{e}_l  W_l \mathbf{e}_l^T \right] \R^T\\
       &= \R \left[ \mathrm{blk}_c^k (X) + \sum_{l\in\Edges\setminus c} \mathrm{blk}_c^k(W_l) \right] \R^T\\
  g(X) &= X^{-1},~ h(X) = \mathbf{tr}\left[ \R^T X \R\right],
\end{align}
where 
\begin{align}
  \mathrm{blk}_l^k(W_l) =  \mathbf{e}_l  W_l \mathbf{e}_l^T = \left( e_l\otimes I_k\right) W_l \left( e_l^T \otimes I_k\right)
\end{align}
denotes the $kn\times kn$ matrix with $W_l$ on the $l$th $k\times k$ diagonal block, with zeros otherwise.
Then, the cost function with $W_c$ as the argument, is given by 
\begin{align}
  \mathrm{tr}\left[ \R^T \left(\R\W\R^T\right)^{-1} \R\right] = \left(h\circ g \circ f\right)(W_c).\label{eq:1}
\end{align}
These functions have differentials 
\begin{align}
  df_X[H] &= \R \left[ \mathrm{blk}_c^k(H) \right] \R^T,~  dg_X[H] = - X^{-1} H X^{-1}\\
  dh_X[H] &= \mathrm{tr}\left[ \R^TH\R \right].
\end{align}
By the chain rule, we have that 
\begin{align}
  d(g\circ f)_X[H] &= - Y^{-1} \R \mathrm{blk}_c^k (H) \R^T Y^{-1}\\
  Y \triangleq f[H] &= \R \left[ \mathrm{blk}_c^k (H) + \sum_{l\in\Edges\setminus c} \mathrm{blk}_l^k(W_l) \right] \R^T\\
                   &\triangleq \R \W_H^c \R^T,
\end{align}
and so 
\begin{align}
      \resizebox{\columnwidth}{!}
    {$
  d(g\circ f)_X[H] = - \left(\R \W_H^c \R^T\right)^{-1} \R \mathrm{blk}_c^k (H) \R^T \left(\R \W_H^c \R^T\right)^{-1}
    $}\\\vspace{-0.5cm}
        \resizebox{\columnwidth}{!}
    {$
  d(h\circ g\circ f)_X[H]= -\mathrm{tr} \left[Q \mathrm{blk}_k^c(H) Q^T \right],~
  Q^T \triangleq \R^T\left( \R \W_H^c \R^T \right)^{-1} \R.
      $}
\end{align}
Hence, we can write 
\begin{align}
   &d(h\circ g\circ f)_X[H]= -\mathrm{tr} \left[Q \mathrm{blk}_k^c(H) Q^T \right]\\
                          &= -\mathrm{tr} \left[Q \mathbf{e}_c  H \mathbf{e}_c^T Q^T \right]
                          = \left \langle -\mathbf{e}_c^T Q^TQ \mathbf{e}_c, H \right \rangle,
\end{align}
and so the gradient of \eqref{eq:1} with respect to the $c$th weight $W_c$ is identified as the $c$th $k\times k$ diagonal block of $-Q^TQ$.

\end{proof}

\textbf{Proposition~\ref{prop:2}.}
\begin{proof}
Consider the cost function (denoted by $f(\epsilon_{i,j}^{-1})$) without the box constraint, and note that the $\Htwo$ portion can be rewritten as a double sum of the same form as the regularization term,
\[
\mathbf{tr}\left(\R^T \lap_{e,s}^\tau  \R\right) = \mathbf{tr}\left(\E^{-1} \lap \right) = \sum_{i=1}^n \sum_{j=1}^k \epsilon_{i,j}^{-1} \text{deg}(\nu_i),
\]
%
%
resulting in,
\[
f(\epsilon_{i,j}^{-1}) = \frac{1}{2} \sum_{i=1}^n \sum_{j=1}^k \epsilon_{i,j}^{-1} \text{deg}(\nu_i) + \frac{h}{2}\sum_{i=1}^n \sum_{j=1}^k (\epsilon_{i,j}^{-1})^{-r}.
\]
Minimizing this cost alone can be achieved by setting its gradient equal to zero,
\begin{align*}
\frac{\partial f}{\partial \epsilon_{i,j}^{-1}} & = \frac{\text{deg}(\nu_i)}{2} - \frac{h r}{2} (\epsilon_{i,j}^{-1})^{-(r+1)}=0,
\end{align*}
implying that,
$\epsilon_i^\ast  = ({\text{deg}(\nu_i)}/{h r})^{{1}/{r+1}}.$
Projecting this result onto the constraint set gives the result.
\end{proof}


\begin{IEEEbiography}
[{\includegraphics[width=1in,height=1.25in,clip]{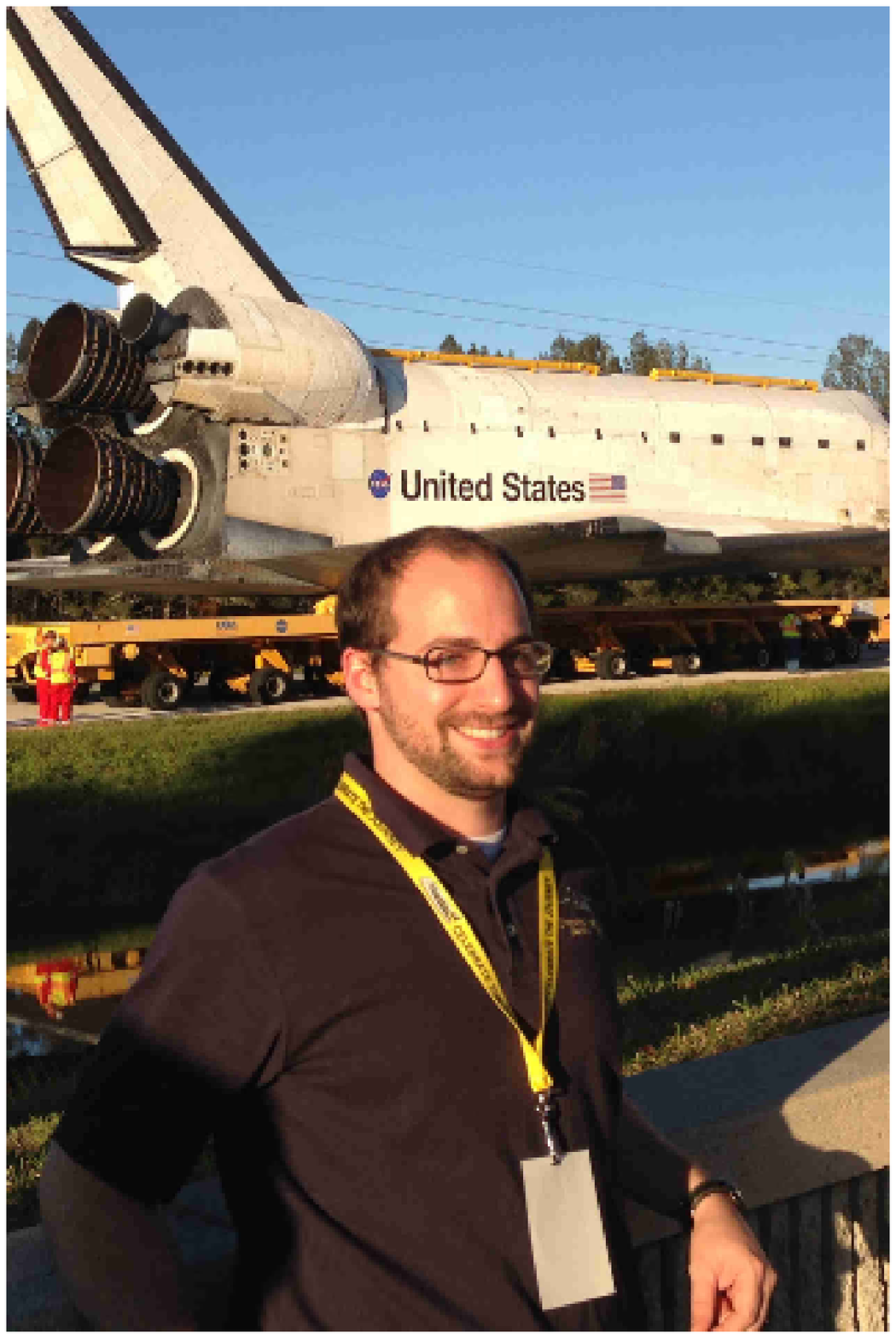}}]
{Dillon Foight} received a B.S. degree in space physics (astrophysics) from Embry-Riddle Aeronautical University, Prescott, AZ, USA in 2009, and worked as a science mission planner for the Chandra Space Telescope until 2015. He is currently a graduate student in the William E. Boeing Department of Aeronautics \& Astronautics at the University of Washington, Seattle, WA, USA. As a member of the Robotics, Aerospace, and Information Networks (RAIN) Lab, he focuses on analysis, control, and influence of networks featuring multi-time scale behavior. He is the recipient of the National Science Foundation Graduate Research Fellowship (2017).
\end{IEEEbiography}
\vskip 0pt plus -1fil
\begin{IEEEbiography}[{\includegraphics[width=1in,height=1.25in,clip]{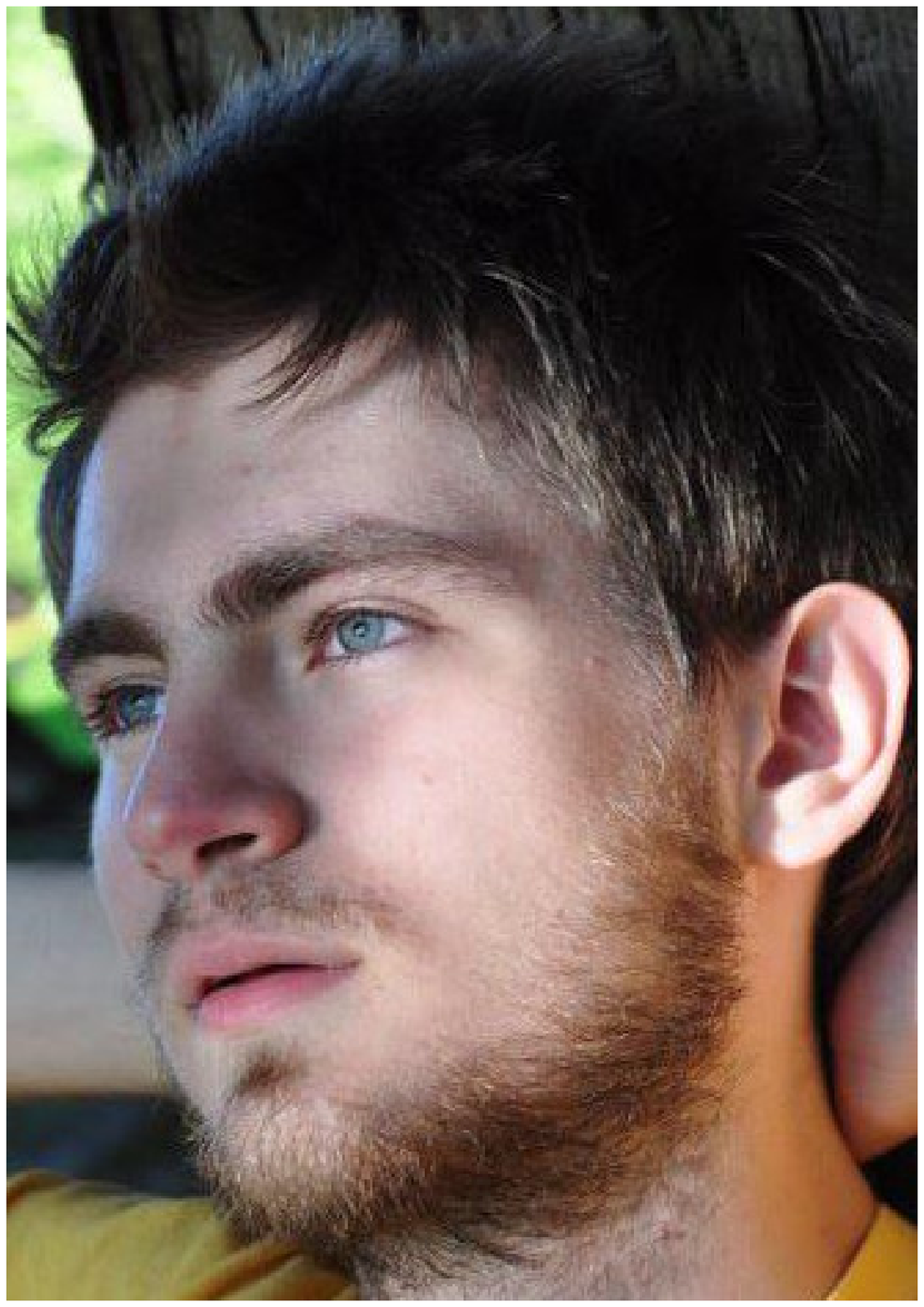}}]
{Mathias Hudoba de Badyn}
is a postdoctoral scholar in the Automatic Control Laboratory at the Swiss Federal Institute of Technology in Z\"{u}rich.
He received his Ph.D.~degree in the William E. Boeing Department of Aeronautics and Astronautics, and an M.Sc.~degree in the Department of Mathematics in 2019 at the University of Washington.
In 2014, he graduated from the University of British Columbia with a BSc in Combined Honours in Physics and Mathematics.
He held an NSERC PGS-D (CGS-D offered) from 2017-2019, and a University of Washington College of Engineering Dean's Fellowship from 2014-2015.
His research interests include the analysis and control of networked dynamical systems, with applications to autonomous air and space multi-vehicle systems.
\end{IEEEbiography}
\vskip 0pt plus -1fil
\begin{IEEEbiography}
[{\includegraphics[width=1in,height=1.25in,clip]{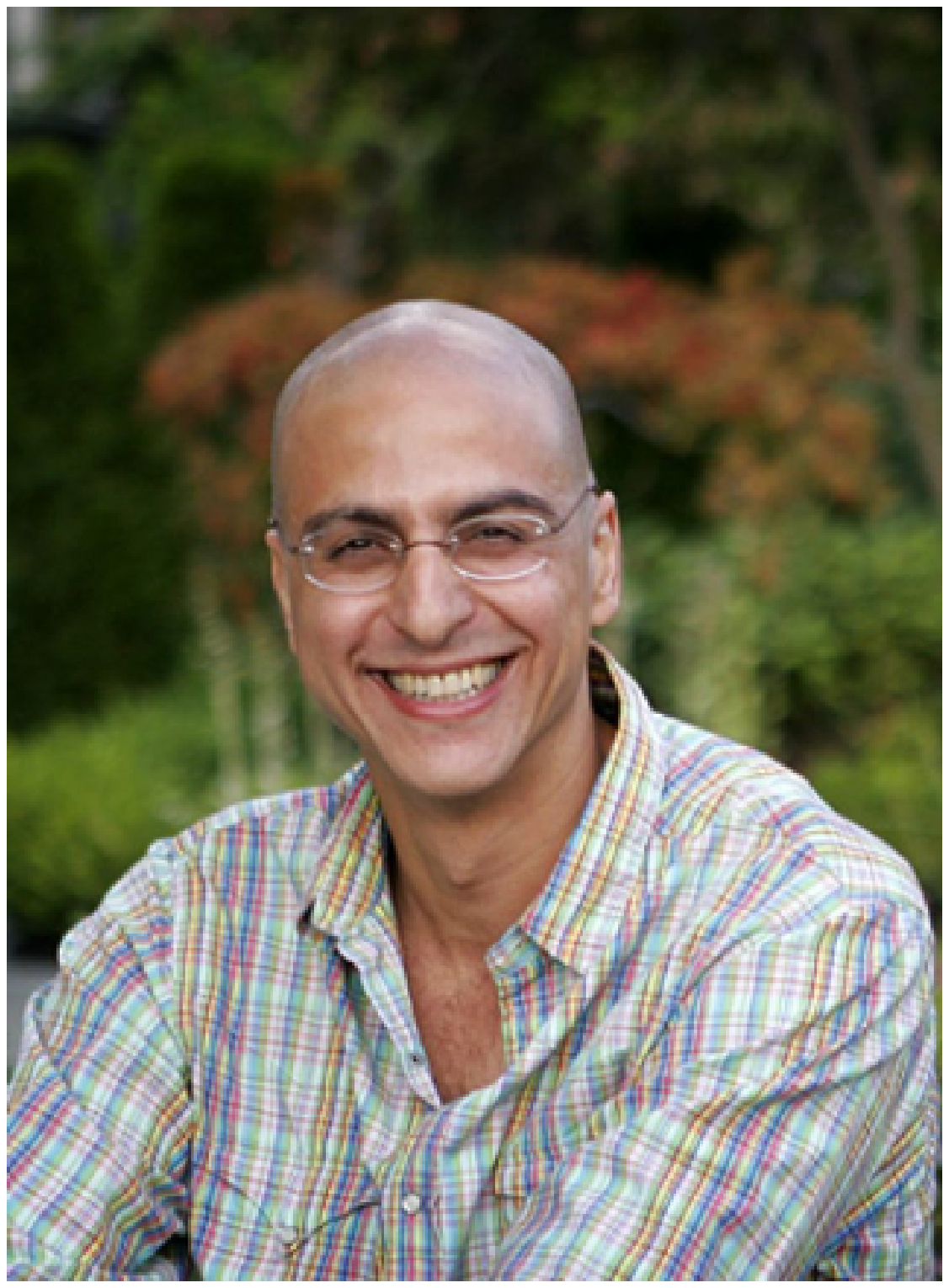}}]
{Mehran Mesbahi}
is a Professor of Aeronautics \& Astronautics and an Adjunct Professor of Mathematics and Electrical Engineering at the University of Washington. 
He received his Ph.D. from USC in 1996. 
He was a member of the Guidance, Navigation, and Analysis group at Jet Propulsion Laboratory from 1996-2000 and an Assistant Professor of Aerospace Engineering and Mechanics at the University of Minnesota from 2000-2002. 
He is currently the Director of the Robotics, Aerospace, and Information Networks (RAIN) Laboratory and the Executive Director of the Joint Center for Aerospace Technology Innovation.
His research interests are distributed and networked systems, autonomous aerospace and robotic systems, and the intersection between data science, networks, autonomy, and control theory, with multi-disciplinary applications.
\end{IEEEbiography}

\end{document}